\documentclass{amsart}

\RequirePackage{color}

\def\smallskip{\vskip\smallskipamount}
\def\medskip{\vskip\medskipamount}
\def\bigskip{\vskip\bigskipamount}

\usepackage{amsmath,amsthm,bm,mathrsfs,amscd,tikz}
\usepackage{ytableau}
\usepackage{comment}
\usepackage{mathdots}
\usepackage[all]{xy}
\usepackage{graphicx}
\usepackage[colorinlistoftodos]{todonotes}
\usepackage{enumerate}
\usepackage{feynmp-auto}
\usepackage[english]{babel}
\usepackage[utf8]{inputenc}
\usepackage{float}
\usepackage{tikz}
\usepackage{tikz-cd}
\usepackage{amssymb}
\usepackage{mathtools}
\usepackage{soul}
\usetikzlibrary{decorations.pathmorphing}

\usepackage[utf8]{inputenc}

\newtheoremstyle{thmstyle}{}{}{\itshape}{}{\bfseries}{ }{5pt}{}
\newtheoremstyle{exstyle}{}{}{}{}{\bfseries}{ }{5pt}{}
\newtheoremstyle{defstyle}{}{}{}{}{\bfseries}{ }{5pt}{}
\newtheoremstyle{remstyle}{}{}{}{}{\bfseries}{ }{5pt}{}

\theoremstyle{thmstyle}
\newtheorem{thm}{Theorem}[section]
\newtheorem{theorem}[thm]{Theorem}
\newtheorem{lem}{Lemma}[section]
\newtheorem{lemma}[thm]{Lemma}

\newtheorem{prop}[thm]{Proposition}
\newtheorem{proposition}[thm]{Proposition}

\newtheorem{definition}[thm]{Definition}

\theoremstyle{exstyle}

\newtheorem{example}[thm]{Example}

\theoremstyle{defstyle}
\newtheorem{defn}{Definition}[section]

\newtheorem{def-prop}[thm]{Definition-Proposition}
\newtheorem{def-lem}[thm]{Definition-Lemma}
\newtheorem{rem-convention}[thm]{Remark-Convention}
\newtheorem{def-note}[thm]{Definition-Notation}

\theoremstyle{remstyle}

\newtheorem{remark}[thm]{Remark}

\theoremstyle{remstyle}

\newcommand{\Hom}{\operatorname{Hom}}

\DeclareMathOperator*{\rad}{rad}

\DeclareMathOperator*{\modu}{mod}
\DeclareMathOperator*{\Modu}{Mod}

\newcommand{\Str}{\operatorname{Str}}

\DeclareMathOperator*{\ind}{ind}
\DeclareMathOperator*{\Ind}{Ind}

\DeclareMathOperator*{\End}{End}

\DeclareMathOperator*{\node}{node}

\makeatletter
\newcommand*{\doublerightarrow}[2]{\mathrel{
  \settowidth{\@tempdima}{$\scriptstyle#1$}
  \settowidth{\@tempdimb}{$\scriptstyle#2$}
  \ifdim\@tempdimb>\@tempdima \@tempdima=\@tempdimb\fi
  \mathop{\vcenter{
    \offinterlineskip\ialign{\hbox to\dimexpr\@tempdima+1em{##}\cr
    \rightarrowfill\cr\noalign{\kern.5ex}
    \rightarrowfill\cr}}}\limits^{\!#1}_{\!#2}}}

\makeatletter
\newcommand{\doublewidetilde}[1]{{%
  \mathpalette\double@widetilde{#1}%
}}
\newcommand{\double@widetilde}[2]{%
  \sbox\z@{$\m@th#1\widetilde{#2}$}%
  \ht\z@=.9\ht\z@
  \widetilde{\box\z@}%
}
\makeatother

\newcommand{\sturmian}[3]{r_{#1,#2}^{#3}}

\title[Infinite string bricks and Sturmian words]{Infinite string bricks and Sturmian words over some gentle algebras}

\author[Mark Deaconu, Kaveh Mousavand, Charles Paquette]{Mark Deaconu, Kaveh Mousavand, Charles Paquette} 

\address{Mark Deaconu: Department of Mathematics and Statistics, Queen's University, Canada}
\email{mark.deaconu@gmail.com}
\address{Kaveh Mousavand: Representation Theory and Algebraic Combinatorics Unit, Okinawa Institute of Science and Technology (OIST), Japan}
\email{mousavand.kaveh@gmail.com}
\address{Charles Paquette: Department of Mathematics and Computer Science, Royal Military College of Canada, Kingston ON, Canada}
\email{charles.paquette.math@gmail.com}

\subjclass [2020]{16G10, 16G20, 16D80, 05A05}
\keywords{String, Brick, Sturmian word, Gentle algebra}

\begin{document}

\maketitle

\begin{abstract}
We study infinite string modules that are bricks over some gentle algebras. In particular, we first give a complete classification of these modules over the double-Kronecker gentle algebra and prove that each family is in bijection with a family of Sturmian (binary) words. We then generalize some of our results to a larger family of gentle algebras.
\end{abstract}

\tableofcontents

\section{Introduction and outline}

In the recent developments of representation theory of finite dimensional algebras and several related domains, bricks have been playing a central role. Recall that a module is said to be a \emph{brick} if its endomorphism algebra is a division algebra. Bricks form a particular set of indecomposable modules and they govern many algebraic and geometric aspects of the representation theory of algebras. However, a full understanding of their behavior over arbitrary algebras is a formidable task (for some recent studies of bricks, see \cite{MP2} and the references therein). Meanwhile, a nice test class of algebras to study bricks is known to be the gentle algebras, particularly because one can use a range of combinatorial tools. In fact, for gentle algebras, the finite dimensional indecomposable modules are completely classified in terms of string modules and band modules; and these are all indecomposables if and only if the gentle algebra is of finite representation-type. 
In contrast, over infinite representation-type gentle algebras, indecomposables are only partially classified (see Section \ref{Section: Combinatorics of gentle algebras}, and the references therein).

Among infinite-dimensional bricks over gentle algebras, there are two natural families. First, there are the generic bricks; namely, those infinite-dimensional bricks which are of finite length over their endomorphism ring. The generic bricks over gentle algebras are in fact in bijection with the brick-bands, that is, bands for which the minimal family of (finite dimensional) modules they induce are bricks (see \cite{MP1} and \cite{BPS}). Another natural family of infinite-dimensional bricks over gentle algebras are those that can be realized as string modules. As we shall see in the following, some gentle algebras admit generic bricks but no infinite string bricks. On the other hand, from the main result of \cite{Se} and the validity of the second brick Brauer-Thrall conjecture for gentle algebras \cite{MP1}, it follows that the existence of an infinite string brick always implies the existence of a generic brick.

The main goal of this paper is to give a complete classification of infinite string bricks over the double-Kronecker gentle algebra. As we show in this work, the relatively nice combinatorics over this algebra allows us to study such infinite dimensional bricks via Sturmian (binary) words. Our construction also allows us to get some infinite string bricks over arbitrary gentle algebras admitting certain configurations. We note that brick-bands over the double-Kronecker gentle algebra (and higher rank generalization of it) have been recently studied in \cite{DL+} through the use of perfectly clustering words, which, in the case of the double-Kronecker gentle algebra, correspond to Christoffel (binary) words. We also explain the relations between these two sets of bricks and their corresponding binary words.

This manuscript is organized as follows. In Section 2, we recall some background material on binary words that we need in this paper; in particular, different characterizations of Sturmian words. In Section 3, we recall some key notions about gentle algebras and their representations and morphisms. In particular, we present some results on infinite string modules and their morphisms, which allows for a better understanding of infinite string bricks. In Section 4, we consider the case of the double-Kronecker gentle algebra. More specifically, we classify all infinite-dimensional string bricks via Sturmian words. Finally, in Section 5, we generalize this to arbitrary gentle algebras admitting some configurations that allow for similar combinatorics to the double-Kronecker gentle case.

\section{Some combinatorics on binary words}
\label{Section: binary words}
\subsection{Generalities on binary words}
A \emph{binary word} is a sequence of characters in a two-letter language. Unless specified otherwise, throughout we will assume the letter set is $\{a,b\}$. Whenever easier, by $\mathfrak{B}$ we denote the set of all binary words. The \emph{length} of a binary word is the number of letters appearing in it, when finite, and is defined to be infinite otherwise. Note that an infinite binary word can mean left-infinite, right-infinite or double-infinite. In particular, a left-infinite binary word can be written as
\[\cdots c_{-3}c_{-2}c_{-1}\]
where $c_i \in \{a,b\}$, for $i \in \mathbb{Z}_{<0}$, while a right-infinite word can be written as
\[c_{1}c_{2}c_3\cdots\]
where $c_i \in \{a,b\}$, for $i \in \mathbb{N}$. Finally, a
double-infinite binary word is written as 
\[\cdots c_{-2}c_{-1}c_0c_1c_2c_3\cdots\]
where $c_i \in \{a,b\}$, for $i \in \mathbb{Z}$. These three sets of infinite words are pairwise disjoint.

For $w \in \mathfrak{B}$, a \emph{subword} of $w$ is a finite or infinite segment of consecutive letters of $w$. In particular, a given word may appear multiple times as a subword of $w$. That being the case, we sometimes refer to an \emph{occurrence} of a word as a subword, particularly to specify that we are interested in a subword in a given position in $w$.
By a \emph{decomposition} of $w$, we mean $w$ can be written as $w=w_1w_2$, where $w_1$ and $w_2$ are non-empty subwords of $w$, and $w_1$ is neither right-infinite nor double-infinite, while $w_2$ is neither left-infinite nor double-infinite. 
In a decomposition $w = w_1w_2$, we naturally call $w_1$ a \emph{starting subword} of $w$ and $w_2$ an \emph{ending subword} of $w$.

For $w \in \mathfrak{B}$, let $w^T$ denote the \emph{transpose} of $w$, which is the binary word obtained from $w$ by writing it in opposite order. If $w = w_1w_2$, then obviously $w^T = w_2^T w_1^T$. In particular, the transpose of a left-infinite (respectively right-infinite) binary word is right-infinite (respectively left-infinite). 
With these standard terminology and notation in place, we now recall some notions that are used in the following.

\begin{defn}
Let $w$ be a binary word. 
\begin{enumerate}
    \item If $w$ is finite, the \emph{Hamming weight of $w$ with respect to letter $b$} (or simply Hamming weight of $w$) is denoted by $h(w)$ and counts the number of $b$'s occurring in $w$. 
    \item The word $w$ is called \emph{balanced} if, for every positive integer $n$, the Hamming weights of any two (finite) subwords of length $n$ of $w$ differ by at most one.  
\end{enumerate}
\end{defn}

Before presenting an example to clarify the above definitions, let us remark that, for a binary word $w \in \mathfrak{B}$, being balanced imposes a strong condition on $w$. In particular, if $w$ is not balanced, there exists a (possibly empty) subword $u$ of $w$ such that $aua$ and $bub$ both appear as subwords of $w$ (see \cite[Lemma 6.7.1]{Fo}). This observation will be helpful when Sturmian words are considered in the following (see Proposition \ref{Prop: For aperiodic, Sturmian is axa, bxb avoiding}).

 \begin{example}
        Consider the right-infinite binary word $w = abababababab\cdots$. 
        We observe that $w$ is periodic, and one can easily check that subwords of even length $n$ all have Hamming weight $n/2$, whereas the subwords of odd length $n$ have Hamming weights $(n-1)/2, (n+1)/2$. Hence, $w$ is balanced.

        In contrast, the finite binary word $w' = babbababaa$ is not balanced. In particular, $aa$ and $bb$ are two subwords of length $2$ that differ in Hamming weights by $2$.

    \end{example}

\begin{defn}
The \emph{complexity function} is the map $p:\mathfrak{B}\times \mathbb{Z}_{>0} \to \mathbb{Z}_{>0}$ which, for every binary word $w$ and each positive integer $n$, gives $p(w,n)$ as the number of distinct subwords of length $n$ in $w$.
\end{defn}

The complexity function is a useful tool in decising whether a binary word is periodic (see Proposition \ref{Prop: periodic via complexity function}). Moreover, it helps to better understand properties of Sturmian words, which will be introduced in the following (see Definition \ref{Definition: Sturmian}).

\begin{defn}\label{Def: Periodic words}
Let $w$ be an infinite binary word. Then 
\begin{enumerate}
    \item $w$ is \emph{periodic} if there exists a finite subword $\sigma$ of $w$ such that $w=\sigma \sigma \sigma \cdots$, or $w= \cdots \sigma\sigma\sigma$, or $w=\cdots \sigma\sigma\sigma \cdots$.
    \item $w$ is \emph{eventually right-periodic} if $w$ admits a decomposition $w = w_1w_2$ with $w_2$ periodic. 
    \item $w$ is \emph{eventually left-periodic} if $w$ admits a decomposition $w = w_1w_2$ with $w_1$ periodic.
    \item 
    We say $w$ is \emph{aperiodic} if none of the above hold.

\end{enumerate}
\end{defn}

As an useful application of the complexity function, in the following proposition we list some alternative ways to determine whether a given one-sided infinite word is eventually periodic.

\begin{prop}[Theorem 1.3.13 in \cite{Lo}]\label{Prop: periodic via complexity function}
Let $w$ be a one-sided infinite binary word in $\mathfrak{B}$. The following are equivalent:
    \begin{enumerate}
        \item $w$ is eventually periodic.
        \item there exists a positive integer $n$ with $p(w,n) \le n$.
        \item there exists an integer $N$ such that $p(w,n) \le N$, for all $n$. 
    \end{enumerate}
\end{prop}

\begin{remark}
Let $w \in \mathfrak{B}$ be a one-sided infinite word which is not eventually periodic. Then, Proposition \ref{Prop: periodic via complexity function} implies that, for any $n \in \mathbb{Z}_{>0}$, the smallest possible value that $p(w,n)$ can attain is $p(w,n)=n+1$. As we shall see later, this is equivalent to $w$ being a Sturmian word (see Definition \ref{Definition: Sturmian}).
\end{remark}

\subsection{Sturmian words and Christoffel words}

 Many definitions recalled below are taken from \cite{Lo}, or appear as slight modifications of them. Hence, for more details, the reader is referred to the section on Sturmian words from the aforementioned reference (also, see \cite{Fo} and the references therein). We start with the following definition, which allows us to give a unified definition of Sturmian and Christoffel words. For a better understanding, compare the following definition with Remark \ref{rem: cutting sequence} and Example \ref{ex: cutting sequence}. A more representation-theoretical comparison between Sturmian and Christoffel words appears at the end of Section \ref{Subsect: Some related combinatorics of words over gentle algebras}.

\begin{definition} 
Let $D$ be an interval in $\mathbb{R}$. For $m > 0$ and $c\in \mathbb{R}$, the binary word over $D$ with respect to the line $y = mx+c$ is defined through the following steps:
\begin{enumerate}
    \item [(I)] On the line $y= mx + c$, consider the sequence of all points $$ \ldots , p_{-1} = (b_{-1}, a_{-1}), p_0 = (b_0, a_0), p_1=(b_1, a_1), p_2 = (b_2, a_2), p_3 = (b_3, a_3), \ldots$$ where $ \cdots < b_{-1} < b_0 < b_1 < b_2 < \ldots$ are in $D$ such that, for each $i$, either $b_i$ or $a_i$ is an integer. 

    \item [(II)] From the sequence $(\ldots, b_{-1}, a_{-1}, b_0, a_0, b_1, a_1, \ldots)$, consider the subsequence that consists of only the integers; denote this subsequence by $S$.

    \item [(III)] Turn $S$ into a binary word in $\mathfrak{B}$, by replacing each $b_i$ with the letter $b$ and each $a_i$ with the letter $a$. 
\end{enumerate}
The resulting word in $\mathfrak{B}$ is denoted $\sturmian{m}{c}{D}$, and it is called the \emph{(lower) cutting word} of the line $y=mx+c$ over the domain $D$. 
Moreover, as the binary word obtained from a special choice of $D$ and $c$ in the above construction, the infinite word $\sturmian{m}{0}{(0,\infty)}$ in $\mathfrak{B}$ is known as the \emph{characteristic word of slope $m$}.
\end{definition}

\begin{remark}
    \label{rem: cutting sequence}

\begin{enumerate}
    \item For the interval $D$ and a line $y=mx+c$, the lower cutting word $\sturmian{m}{c}{D}$ can be interpreted as the word generated by recording each intersection of the line with the integer grid over the interval $D$. With each intersection of a vertical line corresponding to the letter $b$, and each intersection of a horizontal line corresponding to the letter $a$. In particular, crossing a lattice point would correspond to a $b$ and then an $a$ in that order.
    \item One could dually define the upper cutting word for the line $y=mx+c$ over an interval $D$. The rule is very similar, except when a lattice point is crossed. Each intersection of a vertical line corresponding to the letter $b$, each intersection of a horizontal line corresponding to the letter $a$, and each crossing of a lattice point would correspond to an $a$ followed by a $b$.
\end{enumerate}
        
        \end{remark}

        \begin{example} \label{ex: cutting sequence}
Let $D = (0, \infty)$ and consider the line $y=mx+0$, with $m = \frac{2}{1 + \sqrt{5}}$. 
Observe that the line will never go through a lattice point in $D$.  
        
\begin{figure}
    \centering
            \begin{tikzpicture}[scale=0.75]
            \draw[step=1cm,gray,very thin] (0,0) grid (8,5);
            
            \draw[red, thick] (0,0) -- (8,4.944);
            
            \foreach \x/\y in {1/0.618, 2/1.236, 3/1.854, 4/2.472, 5/3.090, 6/3.708, 7/4.326, 8/4.944}{
              \fill[blue] (\x,\y) circle (2pt);
              \node[above left] at (\x,\y) {\small b};
            }
            
            \foreach \x/\y in {1.618/1, 3.236/2, 4.854/3, 6.472/4}{
              \fill[green] (\x,\y) circle (2pt);
              \node[below right] at (\x,\y) {\small a};
            }
            
            \draw[->] (0,0) -- (8.5,0) node[right] {$x$};
            \draw[->] (0,0) -- (0,5.5) node[above] {$y$};
            
            \end{tikzpicture}
    \caption{Geometric presentation of lower cutting word $\sturmian{m}{0}{D}$ for $D = (0, \infty)$ and $m = \frac{2}{1 + \sqrt{5}}$.}
    \label{Fig: lower cutting}
\end{figure}

        Then, we have $\sturmian{m}{0}{D} = babbababbabb \cdots$, known as the \emph{golden string}. Note that if we change the domain $D$ to $D' = [0, \infty)$, then the above word would have a $ba$ appended to its left, so that it becomes $\sturmian{m}{0}{D'} = bababbababbabb \cdots$. If we instead take the upper cutting word over $[0, \infty)$, we get the binary word $abbabbababbabb \cdots$ 
   \end{example}
   
\begin{remark}
    \begin{enumerate}
        \item Note that if $m, c$ and $1$ are $\mathbb{Q}$-linearly independent, then the line $y=mx+c$ does not go through any lattice points. That being the case, it is evident that, for each $i$, exactly one of $b_i, a_i$ is an integer. 
        \item Observe that if $m < 1$, then the subword $aa$ never occurs in $\sturmian{m}{c}{D}$. Similarly, if $m > 1$, then the subword $bb$ never occurs in $\sturmian{m}{c}{D}$.
        \item Obviously, the line $y=mx+c$ goes through at least two lattice points if and only if $m$ is rational. If this is the case and $D$ is an infinite interval, then $\sturmian{m}{c}{D}$ is either periodic or eventually periodic.
        
    \end{enumerate}
\end{remark}

\begin{definition}
    Let $w$ be an infinite binary word in $\mathfrak{B}$. Then $w$ is said to be \emph{Sturmian} if $w$ is both balanced and aperiodic.
\end{definition}

We now list some alternative characterizations of Strumian words that are useful in the following. For a comprehensive study of Sturmian words and their connections with the rotation sequences, continued fractions and dynamical systems, we refer to \cite[Chapter 6]{Fo}.
    \begin{proposition}\label{Definition: Sturmian}
        (Theorem 2.1.13 in \cite{Lo}) For an infinite length word $w$ in $\mathfrak{B}$, the following are equivalent:
        \begin{enumerate}
            \item $w$ is Sturmian.
            \item $w$ is aperiodic and for the complexity function we have $p(w,n) = n+1$, for all positive integers $n$.
            \item 
            There exist a infinite interval $D$ and an irrational number $m$, such that $w$ is a upper or lower cutting word of the line $y=mx+c$ over $D$; that is $w = \sturmian{m}{c}{D}$, or $w$ is the upper cutting version of it.
        \end{enumerate}
    \end{proposition} 

    \begin{remark}
\begin{enumerate}
    \item If $w$ is one sided-infinite, then $p(w,n)=n+1$ implies that $w$ is not eventually periodic. Indeed, for each one-sided infinite word $w$ that is eventually periodic, $p(w,n)$ is globally bounded. On the other hand, we note that for the double-infinite word $\cdots aaabbb \cdots$, we have $p(w,n)=n+1$, but this binary word is clearly not aperiodic (see Definition \ref{Def: Periodic words}).
    \item Every right-infinite Sturmian word $\sturmian{m}{c}{D}$ has a unique corresponding characteristic Sturmian word $\sturmian{m}{0}{(0,\infty)}$, often called \emph{the characteristic Sturmian word of the same slope}. 
\end{enumerate}

    \end{remark}

We now introduce another important type of binary words which relate to the study of Sturmian words; see Proposition \ref{link Sturmian and Christoffel}. In Section \ref{Subsect: Some related combinatorics of words over gentle algebras}, we further elaborate on the conceptual connections between these two types of binary words in the context of the representation theory of some gentle algebras. 
We encourage the reader to compare the definition with Remark \ref{Remark: Christoffel words} and Example \ref{Example: Christoffel words}.

\begin{definition}\label{Def: Christoffel word}
    Let $m$ be a positive rational number expressed in reduced form $m = p/q$; that is $\rm{gcd}(p,q)=1$.  
    Starting with the domain $D = (0,q)$ and the cutting word $\sturmian{m}{0}{D}$, the \emph{Christoffel word} $m$ is defined as $c_m:=b(\sturmian{m}{0}{D})a$. 
\end{definition}

\begin{remark}\label{Remark: Christoffel words}
    The Christoffel word $c_m$ introduced in the preceding definition can also be described as follows. First, draw a lattice path from $(0,0)$ to $(q,p)$ which is under the line segment from $(0,0)$ to $(q,p)$ but is as close as possible to this line (see Example \ref{Example: Christoffel words}). Reading the lattice path from $(0,0)$ to $(q,p)$, any horizontal unit segment is replaced by $b$, and any vertical unit segment is replaced by $a$. Observe that every such word necessarily starts with the letter $b$ and ends with the letter $a$.
\end{remark}

\begin{example}\label{Example: Christoffel words}
Consider the domain $D = (0,8)$, and let $m = 5/8$. Then, $\sturmian{5/8}{0}{(0,8)} = babbababbab$. Therefore, the corresponding Christoffel word is $c_m = bbabbababbaba$. Figure \ref{Fig: Christoffel figure} shows how to interpret this word as a lattice path staying as close as possible but below the line segment without touching it, except at the extremities.  

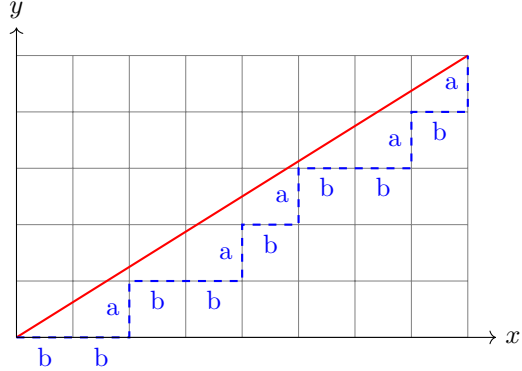
\begin{figure}
    \centering
\begin{tikzpicture}[scale=0.75]
    \draw[step=1cm,gray,very thin] (0,0) grid (8,5);
    
    \draw[red, thick] (0,0) -- (8,5);
    
    \draw[blue, thick, dashed] (0,0) 
        -- (1,0) node[midway,below] {b}  
        -- (2,0) node[midway,below] {b}  
        -- (2,1) node[midway,left] {a}   
        -- (3,1) node[midway,below] {b}  
        -- (4,1) node[midway,below] {b}  
        -- (4,2) node[midway,left] {a}   
        -- (5,2) node[midway,below] {b}  
        -- (5,3) node[midway,left] {a}   
        -- (6,3) node[midway,below] {b}  
        -- (7,3) node[midway,below] {b}  
        -- (7,4) node[midway,left] {a}   
        -- (8,4) node[midway,below] {b}  
        -- (8,5) node[midway,left] {a};  
    
    \draw[->] (0,0) -- (8.5,0) node[right] {$x$};
    \draw[->] (0,0) -- (0,5.5) node[above] {$y$};
\end{tikzpicture}
    \caption{Geometric presentation of Christoffel words}
    \label{Fig: Christoffel figure}
\end{figure}

\end{example}

The following facts are well known, and they are equivalent to Proposition 2.2.24 and Corollary 2.2.28 in \cite{Lo}. 
    
 \begin{prop} \label{link Sturmian and Christoffel}
     \begin{enumerate}
        \item (Prop 2.2.24. \cite{Lo}) If $w$ is a characteristic word, then there is an infinite sequence of Christoffel words $bw_1a, b w_2 a, \ldots$ such that $w_1, w_2, \ldots$ are all the starting subwords of $w$.
         \item (Corollary 2.2.28 \cite{Lo}) If $bwa$ is a Christoffel word, then there exists a Sturmian word $w'$ having $w$ as a starting subword.
     \end{enumerate}
\end{prop}

We will also need the following facts on Sturmian words.

\begin{proposition}[Prop. 2.1.3 in \cite{Lo}] \label{Prop: For aperiodic, Sturmian is axa, bxb avoiding}
    Let $w$ be an infinite aperiodic word. Then $w$ is Sturmian if and only if, for any subword $x$ of $w$, $axa$ and $bxb$ are not both subwords of $w$.
\end{proposition}

\begin{proposition}[Prop. 2.1.22 in \cite{Lo}] \label{sturmian with lattice point}
    Let $w$ be a right-infinite Sturmian word in $a,b$ with the property that both $aw, bw$ are Sturmian. Then $w = \sturmian{m}{0}{(0,\infty)}$, with $m$ irrational.
\end{proposition}

\begin{prop}[Prop 2.1.23 in \cite{Lo}] \label{Both sa and sb}
    Let $w$ be a right-infinite Sturmian word and let $s$ be a finite word. Consider $w'$ the corresponding characteristic Sturmian word of the same slope. Then $s$ is such that both $sa, sb$ are subwords of $w$ if and only if $s^T$ is a starting subword $w'$.
\end{prop}

\section{Combinatorics of gentle algebras}\label{Section: Combinatorics of gentle algebras}
\subsection*{Notations and Conventions}

Throughout, $k$ denotes an algebraically closed field, $A$ is always assumed to be a finite dimensional associative $k$-algebra with a multiplicative identity. By $\modu A$ we denote the category of all finitely generated right $A$-modules, whereas $\Modu A$ denotes the category of all right $A$-modules.
We let $\ind(A)$ and $\Ind(A)$ denote the set of all isomorphism classes of indecomposable modules, respectively in $\modu A$ and $\Modu A$.

By a \textit{quiver} we always mean a finite and connected directed graph, formally given by a quadruple $Q=(Q_0,Q_1,s,t)$, consisting of two finite sets $Q_0$ and $Q_1$, together with two functions $s, t:Q_1\to Q_0$. Elements of $Q_0$ and $Q_1$ are respectively called \textit{vertices} and \textit{arrows} of $Q$. Moreover, for $\alpha:i \rightarrow j$ in $Q_1$, the vertex $s(\alpha)=i$ is its \textit{source} and $t(\alpha)=j$ is its \textit{target}. Unless specified otherwise, we typically use lower case Greek letters $\alpha$, $\beta$, $\gamma$, $\ldots$ to denote arrows of $Q$. 

A \emph{path of length} $d\geq 1$ in $Q$ is a finite sequence of arrows $\gamma_{1}\gamma_{2}\cdots\gamma_{d}$ such that $s(\gamma_{j+1})=t(\gamma_{j})$, for every $1 \leq j \leq d-1$. Moreover, to each vertex $i \in Q_0$ we also associate a path of length $0$, denoted $e_i$, called the \textit{lazy path} (or stationary path) at vertex $i$. For each $i \in Q_0$, we have $s(e_i) = t(e_i) = i.$ The \emph{path algebra} of $Q$, denoted by $kQ$, is generated by the set of all paths (including the lazy paths) as a $k$-vector space. The multiplication in $kQ$ is induced by the concatenation of paths and is extended to $kQ$ by bilinearity. By $R_Q$ we denote the \emph{arrow ideal} in $kQ$, that is, the two-sided ideal generated by all elements of $Q_1$. A two-sided ideal $I \subseteq kQ$ is called \emph{admissible} if $R_Q ^m \subseteq I \subseteq R_Q^2$, for some $m\geq 2$. Consequently, $kQ/I$ is always a basic and connected finite dimensional algebra. In fact, every basic and connected finite dimensional algebra $A$ over an algebraically closed field $k$ admits a presentation of this form. For the standard notations and terminology, as well as the well-known rudiments of the representation theory of associative algebras, we refer to \cite{ASS}.

\subsection{Gentle algebras}
Here we recall some standard notions in the context of gentle algebras, freely used in the following. First, recall that a finite dimensional algebra $A = kQ/I$ is said to be a \emph{gentle algebra} if the following hold:

\begin{enumerate}
\item At each vertex $v \in Q_0$, there are at most two incoming arrows to $v$ and at most two outgoing arrows from $v$.
\item For each arrow $\alpha \in Q_1$, there is at most one $\beta \in Q_1$ starting at $t(\alpha)$ such that $\alpha\beta \in I$, and also at most one $\gamma\in Q_1$ starting at $t(\alpha)$ such that $\alpha\beta \not\in I$.

\item For each arrow $\alpha \in Q_1$, there is at most one $\gamma \in Q_1$ ending at $s(\alpha)$ such that $\gamma\alpha \in I$, and also at most one $\epsilon\in Q_1$ ending at $s(\alpha)$ such that $\epsilon\alpha\not\in I$.

\item There is a set of paths of length two that generate $I$.
\end{enumerate}

Observe that if $\alpha \in Q_1$ is a loop that starts and ends at vertex $v \in Q_0$, then it is considered both as an incoming arrow to $v$, as well as an outgoing arrow from $v$. In particular, if $A=kQ/I$ is a gentle algebra and $\alpha$ is a loop in $Q$, we must always have $\alpha\alpha \in I$, simply denoted by $\alpha^2 \in I$.

Due to the combinatorial nature of gentle algebras, they are often presented in terms of explicit quivers and relations. For a gentle algebra $A=kQ/I$, a set of quadratic relations that generate the ideal $I$ is often depicted by some dotted lines on the two arrows whose composition belongs to $I$, as shown in the example below.

\begin{example}\label{Example of gentle algebra}
Let $A=kQ/I$, where $Q$ is the following quiver, and $I=\langle \alpha\beta, \delta\gamma 
\rangle$. Note that the quadratic relation $\alpha \beta$ and $\delta\gamma$ are depicted by the dotted lines on the corresponding composition of arrows in $Q$.
It is easy to check $A=kQ/I$ satisfies all the conditions listed for gentle algebras.

\begin{figure}[!htbp]
\begin{center}
\begin{tikzpicture}[scale=0.6]
\draw [->] (-2.2,0) --(-0.4,0);
\node at (-1.3,0.3) {$\alpha$};
\node at (-2.5,0) {$^1\bullet$};
\draw [->] (0.2,0) --(1.8,0);
\node at (1,0.3) {$\beta$};
\node at (-0.2,0.1) {$\bullet^2$};
\draw [->] (-0.2,-1.9) -- (-0.2,-0.2);
\node at (2.1,0.1) {$\bullet^3$};
\draw [->] (2.2,0) --(4,0);
\node at (3,0.3) {$\gamma$};
\node at (4.3,0.1) {$\bullet^4$};

\node at (-0.3,-2) {$_6\bullet$};
\node at (-0.4,-1) {$\theta$};
\draw [->] (2,-2) --(2,-0.2) ;
\node at (2.1,-2) {$\bullet_5$};
\node at (2.3,-1) {$\delta$};
\draw [->] (1.9,-2) --(0,-2);
\node at (1,-2.3) {$\epsilon$};

-------
\draw[] (-.2,0)--(0.2,0);
  \draw [thick, dashed] [color=red] (-1,0) to [bend left=50] (0.6,0);
  \draw [thick, dashed] [color=red] (2.1,-0.75) to [bend left=50] (3,-0.1);
\end{tikzpicture}
\end{center}
\caption{A bound quiver $(Q,I)$, where $A = kQ/I$ is gentle. The dashed arcs depict the generators of $I$, given by $\alpha\beta$ and $\delta\gamma$.}
\label{figure_gentle_quiver}
\end{figure}
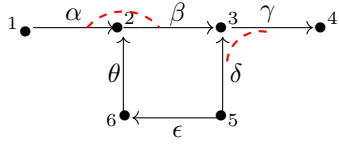
\end{example}

\subsection{Strings and bands}
Let $A=kQ/I$ be a gentle algebra, with a fixed set of quadratic relations that generate $I$. By $Q_1^{-1}$ we denote the set of formal inverses of arrows of $Q$. In particular, elements of $Q_1^{-1}$ are denoted by $\gamma^{-1}$, for $\gamma \in Q_1$, such that $s(\gamma^{-1}):=t(\gamma)$ and $t(\gamma^{-1}):=s(\gamma)$. 
A \emph{string} of length $d\geq 1$ in $A$ is a word $w = \gamma_1^{\epsilon_1}\cdots\gamma_d^{\epsilon_d}$, with $\gamma_i \in Q_1$ and $\epsilon_i \in \{+1, -1\}$, for all $i \in \{1,2,\cdots,d\}$, which satisfies the following conditions:

\begin{enumerate}
\item[(P1)] $s(\gamma_{i+1}^{\epsilon_{i+1}})=t(\gamma_i^{\epsilon_i})$ and $ \gamma_{i+1}^{\epsilon_{i+1}} \neq \gamma_i^{-\epsilon_i}$, for all $i\in \{1,\cdots,d-1 \}$;
\item[(P2)] $w$ and also $w^{-1} := \gamma_d^{-\epsilon_d}\dots\gamma_1^{-\epsilon_1}$ do not contain a subpath in $I$.
\end{enumerate}

If $w = \gamma_1^{\epsilon_1}\cdots\gamma_d^{\epsilon_d}$ and $\epsilon_i = +1$ for some $i$, we write $\gamma_i$ rather than $\gamma_i^{+1}$. We say $w$ \textit{starts} at $s(w)=s(\gamma_1^{\epsilon_1})$ and \textit{terminates} at $t(w)=t(\gamma_d^{\epsilon_d})$. We also associate a zero-length string to every vertex $i \in Q_0$, which we denote by $e_i$.

For a given gentle algebra $A$, observe that strings of positive length are nothing but some words with letters in $Q_1 \sqcup Q_1^{-1}$ that satisfy certain conditions specified above. 
Consequently, those definitions from Section \ref{Section: binary words} on words which do not rely on the fact that the words are binary can also be adapted to strings. For instance, in the context of strings over gentle algebras, one can naturally consider left-infinite, right-infinite and double-infinite strings, as well as the notion of periodicity (see Definition \ref{Def: Periodic words}). 
More specifically, when a string is formed using the concatenation of two finite strings, say $a,b$ such that $ab, ba, a^2, b^2$ are again strings of $A$, then we can use the other concepts from Section \ref{Section: binary words} on binary words, including Hamming weights and complexity. This transition will become important in the following.

As remarked above, we define \emph{one-sided infinite} and \emph{double-sided infinite} strings in $A$. More precisely, $w = \gamma_1^{\epsilon_1}\gamma_2^{\epsilon_2}\cdots$ is a \emph{right-infinite} string if for each pair of positive integers $i$ and $j$ with $i<j$, the subword $\gamma_i^{\epsilon_i}\cdots\gamma_j^{\epsilon_j}$ of $w$ is a string in $A$. That being the case, $w$ starts at $s(\gamma_1^{\epsilon_1})$ and never terminates.
The left-infinite strings are defined dually. In particular, for a right-infinite string $w = \gamma_1^{\epsilon_1}\gamma_2^{\epsilon_2}\cdots$, we have $w^{-1} := \cdots \gamma_2^{-\epsilon_2}\gamma_1^{-\epsilon_1}$ is a left-infinite string in $A$.
A word $w$ in $Q_1 \sqcup Q_1^{-1}$ is called a \emph{double-infinite string} if it can be written as $w = w_Lw_R$ where $w_L$ is left-infinite and $w_R$ is left-infinite. Obviously, a double-infinite string neither has a start nor an end vertex, and each (consecutive) finite subword of it is a string in $A$.
Let $\Str(A)$ denote the set of all strings in $A$ (including the one-sided infinite and double-infinite strings).

Let $w$ belong to $\Str(A)$. For $\gamma_i \in Q_1$, if $\gamma_i^{\epsilon_i}$ appears in $w$, then it is called a \emph{direct} arrow in $w$ if $\epsilon_i=+1$, and is called an \emph{inverse} arrow in $w$ if $\epsilon_i=-1$. Consequently, $w$ is called \emph{direct string} if every arrow in it is a direct arrow. Inverse strings are defined dually. Because $A$ is a finite dimensional algebra, $\Str(A)$ contains only finitely many direct strings (similarly inverse strings), and all such strings are finite, with a global bound on the length of direct and inverse strings in $\Str(A)$.

For a gentle algebra $A$, and a finite string $w \in \Str(A)$ of positive length, we say $w$ is a \emph{cyclic} string if $s(w) = t(w)$. A cyclic string $w$ is called a \emph{band} provided $w^2$ is a string and $w$ is not a power of a string of a strictly smaller length.
If $w = \gamma_1^{\epsilon_1}\cdots\gamma_d^{\epsilon_d}$ is a band in $\Str(A)$, then obviously $w$ is neither a direct string nor an inverse string. In particular, there exists $1 \leq i < j \leq d$ such that $\epsilon_i=-\epsilon_j$.
Due to the reasons explained below, we always identify $w$ with it is inverse $w^{-1}:= \gamma_d^{-\epsilon_d}\cdots\gamma_1^{-\epsilon_1}$, as well as with all possible cyclic permutations of $w$ and their inverses. In other words, each band in $\Str(A)$ is considered up to the cyclic permutations and their inverse. We remark that, even after the above-mentioned identifications, a gentle algebra can still admit infinitely many distinct bands (see Section \ref{section: double-Kronecker algebra}).

\medskip

The above definitions, together with a simple combinatorial argument, imply the following handy lemma.

\begin{lem}
Let $A$ be a gentle algebra. Then, $\Str(A)$ contains infintiely many strings if and only if $\Str(A)$ contains a band.
\end{lem}
In the following example, we clarify some of the notions introduced above.

\begin{example}
Let $A=kQ/I$ be the gentle algebra given in Example \ref{Example of gentle algebra}. Observe that $w =\beta \delta^{-1}\epsilon \theta \alpha^{-1}$ is a string in $\Str(A)$. 
Moreover, $w^{-1}:=\alpha \theta^{-1}\epsilon^{-1} \delta \beta^{-1}$ also belongs to $\Str(A)$, which is the formal inverse of $w$.
The longest direct string (respectively, inverse string) in $\Str(A)$ is $\epsilon \theta \beta \gamma$ (respectively, $\gamma^{-1}\beta^{-1}\theta^{-1}\epsilon^{-1}$).

Consider $u=\beta \delta^{-1}\epsilon \theta$ in $\Str(A)$. It is easy to check that $u$ is a band in $\Str(A)$. In particular, $u$ is a cyclic string and $u^m \in \Str(A)$, for every $m\in \mathbb{Z}_{>0}$, and $u$ is not a power of any smaller string. Observe that, using $u$ or $u^{-1}$, we can generate arbitrarily long strings in $\Str(A)$, and $u^{m}$ is a string of length $4m$. Hence, $\Str(A)$ contains infinite strings, as well as infinitely many strings of finite lengths. 

In fact, for each pair of distinct vertices $i$ and $j$ in $Q_0$, there exist infinitely many strings which start at $i$ and end at $j$. Moreover, we can explicitly generate a right-infinite string $a:=uu\cdots$, a left-infinite string $b:=\cdots uu$, as well as a double-sided infinite string $c:=\cdots uuu \cdots$. Obviously, $a$, $b$ and $c$ are periodic. 
However, we can also generate a non-periodic left-infinite string $\alpha v v\cdots$, where $v:=\theta^{-1} \epsilon^{-1} \delta \beta^{-1}$, which is eventually periodic. Similarly, we can generate non-periodic right-infinite string in $\Str(A)$ which is eventually periodic. In fact, in this example, all infinite strings in $\Str(A)$ are eventually periodic (but not necessarily periodic).

Observe that $v$ is the inverse of a cyclic permutation of the band $u$. Hence, although $u$ and $v$ specify two distinct strings in $\Str(A)$, they give the same band. In fact, up to the cyclic permutations of $u$ and their inverses, $\Str(A)$ contains a unique band, but infinitely many distinct finite cyclic strings.
\end{example}

\subsection{String modules and their homomorphisms} \label{Subsection: Strings}
Let $A$ be a gentle algebra and $w\in \Str(A)$ a finite string. Then we can formally write $w = \gamma_1^{\epsilon_1}\cdots \gamma_d^{\epsilon_d}$. The corresponding walk on $Q$ determined by $w$ can be expressed as the sequence $$\xymatrix{x_{1} \ar@{-}^{\gamma_1}[r] & x_2 \ar@{-}^{\gamma_{d-1}}[r] \cdots & x_{d}\ar@{-}^{\gamma_d}[r] & x_{d+1}}.$$
Here, $x_1,\ldots, x_{d+1}$ denote the vertices visited by $w$, \emph{a priori} multiple times, and the orientation of each arrow $\gamma_i$ is suppressed in this notation. To $w$, we associate the \textit{string module} given by the representation $ M(w) := ((V_i)_{i \in Q_0}, (\varphi_\alpha)_{\alpha\in Q_1})$, where, for each $i\in Q_0$, the vector space $V_i$ is
$$\begin{array}{cccccccccccc}
V_i & := & \left\{\begin{array}{ccl} \displaystyle\bigoplus_{j: x_j = i}kx_j & & \text{if } i = x_j \text{ for some } j \in \{1,\ldots, d+1\}\\ 0 & & \text{otherwise} \end{array}\right.
\end{array}$$
and for each $\alpha \in Q_1$, the linear transformation $\varphi_\alpha:V_{s(\alpha)}\rightarrow V_{t(\alpha)}$ is given with respect to the bases of $V_{s(\alpha)}$ and $ V_{t(\alpha)}$, as following:
$$\begin{array}{cccccccccccc}
\varphi_\alpha(x_k) & := & \left\{\begin{array}{lcl} x_{k-1} & & \text{if } \alpha = \gamma_{k-1} \text{ and } \epsilon_k = -1\\ x_{k+1} & & \text{if } \alpha = \gamma_{k} \text{ and } \epsilon_k = 1\\  0 & & \text{otherwise} \end{array}\right.
\end{array}$$
Observe that, $\dim_k(V_i) = \#\{j \in \{1,\ldots, d+1\} \,|\, \ x_j = i\}$, for any $i \in Q_0$. Moreover, for any string $w \in \Str(A)$, we always have that $M(w) \cong M(w^{-1})$ as representations of $Q$. In other words, viewed as $A$-modules, $M(w)$ and $M(w^{-1})$ are isomorphic. 
As explained in the next subsection, we can extend the above definition to associate an indecomposable module to each infinite string.

We further remark that, to every band in $\Str(A)$, one can associate a family of indecomposable $A$-module, the so-called band modules. For the explicit construction of band modules, we refer to \cite{BR}. In fact, for any gentle algebra $A$, all of the indecomposable modules in $\modu A$ are given by string modules or band modules, and these two subfamilies are known to be disjoint. We emphasize that each band $w$ in $\Str(A)$ gives rise to an infinite family of (non-isomorphic) indecomposable $A$-modules of the same dimension, as well as an infinite family of string modules given by substrings of $w^m$, where $m\in \mathbb{Z}_{>0}$.
Because band modules are not used in this work, we do not recall an explicit definition of them.

For each $w \in \Str(A)$, the \emph{diagram} of $w$ is a pictorial presentation of the string module $M(w)$ that consists of a sequence of up and down arrows, drawn from left to right. 
In particular, starting from vertex $s(w)$, for every direct arrow we put a right-down arrow outgoing from the current vertex, whereas for each inverse arrow we put a left-down ending at the current vertex. These notions, as well as the construction of a string module, are illustrated in the following example.

\begin{example}
Let $A=kQ/I$ be the same algebra as in Example \ref{Example of gentle algebra}. The diagram of the string $w =\beta \delta^{-1}\epsilon \theta \alpha^{-1}$ in $\Str(A)$ and the module $M(w)$ are as follows:
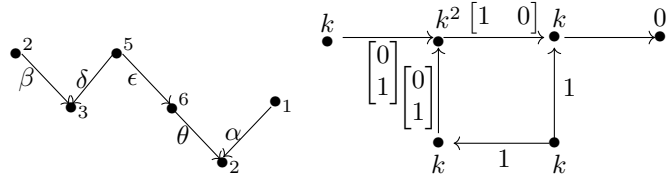
\begin{center}
\begin{tikzpicture}[scale=0.67]
-------------------------
-------------------------
\node at (-5,0.1) {$\bullet^2$};
\draw [->] (-5,0) -- (-4.05,-1);
\node at (-4.9,-0.5) {$\beta$};
--
\node at (-3.9,-1.12) {$\bullet_3$};
\draw [->] (-3.1,0.1) -- (-4,-1);
\node at (-3.8,-0.5) {$\delta$};
\node at (-3,0.1) {$\bullet^5$};
--
\draw [->] (-3,0) -- (-2.05,-1);
\node at (-2.8,-0.5) {$\epsilon$};
\node at (-1.9,-1) {$\bullet^6$};
--
\draw [->] (-2,-1.1) -- (-1.05,-2.1);
\node at (-1.8,-1.6) {$\theta$};
\node at (-0.9,-2.2) {$\bullet_2$};
--
\draw [->] (-0.05,-1.05) -- (-1,-2.05);
\node at (-0.8,-1.6) {$\alpha$};
\node at (0.15,-1) {$\bullet_1$};
---------------------------------
\end{tikzpicture}
\begin{tikzpicture}[scale=0.67]
\draw [->] (5.8,0) --(7.6,0);
\node at (6.6,-.7) {$\begin{bmatrix}
    0 \\
    1 \\
    \end{bmatrix}$};
\node at (5.5,-0.1) {$\bullet$};
\node at (5.5,0.3) {$k$};
\draw [->] (7.8,0) -- (9.8,0);
\node at (9,0.4) {$\begin{bmatrix}
    1 & 0\\
    \end{bmatrix}$};
\node at (7.7,-0.1) {$\bullet$};
\node at (7.9,0.4) {$k^2$};
\draw [->] (7.7,-1.9) -- (7.7,-0.2);
\node at (7.7,-2.1) {$\bullet$};
\node at (7.7,-2.5) {$k$};
\node at (7.3,-1.2) {$\begin{bmatrix}
    0 \\
    1 \\
    \end{bmatrix}$};
\draw [->] (10,-2) --(10,-0.2) ;
\node at (10,-2.1) {$\bullet$};
\node at (10.1,-2.5) {$k$};
\node at (10.3,-1) {$1$};
\draw [->] (9.9,-2.1) --(8,-2.1);
\node at (9,-2.4) {$1$};
\node at (10,0) {$\bullet$};
\node at (10.1,0.4) {$k$};

\draw [->] (10.2,0) --(12,0) ;
\node at (12.1,0) {$\bullet$};
\node at (12.1,0.4) {$0$};
\end{tikzpicture}
\end{center}
\end{example}

\begin{definition}\label{Def: general envelope}
    Let $A$ be a gentle algebra, $s, w \in \Str(A)$ with $s$ a substring of $w$.
    \begin{enumerate}
        \item Given an occurrence of $s$ as a substring of $w$, we define the \emph{envelope} ${\rm env}_w(s)$ of $s$ in $w$ to be the unique substring $s'$ of $w$ containing $s$ with the property that $s'$ is obtained from $s$ by appending one letter to the left of $s$ whenever $s$ is not a left end of $w$; and appending one letter to the right of $s$ whenever $s$ is not a right end of $w$.
        
        \item We define ${\rm Env}_w(s)$ to be the set of all envelopes ${\rm env}_w(s')$ where $s'$ runs through all possible occurrences of $s$ as a substring in $w$.
    \end{enumerate}
    
\end{definition}

The following remarks clarifies some of the above points.

\begin{remark}
\begin{enumerate}
    \item If $s=w$ and $s$ is the whole string $w$, then ${\rm env}_w(s)=w$.
    \item Let $s$ be a finite and a proper substring of $w$. Then, the length of ${\rm env}_w(s)$ is exactly one more than the length of $s$ if and only if $s$ is a starting or ending proper substring of $w$.
\end{enumerate}
    
\end{remark}

The next lemma clarifies our motivations for considering the notion of envelope. 

\begin{lemma} \label{Lemma envelopes}
     Let $A$ be a gentle algebra and $w$ a string with $s$ an occurence of some finite substring of $w$. Consider the envelope ${\rm env}_w(s) = \mu s \nu$. Then
     \begin{enumerate}
         \item The string module $M(s)$ is a submodule of $M(w)$ if and only if $\mu$ is an arrow or is trivial; and $\nu$ is the inverse of an arrow or is trivial. In this case we say ${\rm env}_w(s)$ is a submodule envelope, and the occurrence $s$ is called a submodule substring of $w$.
         \item The string module $M(s)$ is a quotient module of $M(w)$ if and only if $\mu$ is the inverse of an arrow or is trivial; and $\nu$ is an arrow or is trivial. In this case we say ${\rm env}_w(s)$ is a quotient envelope, and the occurrence $s$ is called a quotient substring of $w$.
     \end{enumerate}

\end{lemma}

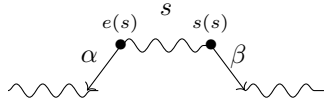
\begin{figure}[h]
\begin{center}
\begin{tikzpicture}[scale=0.6]
\draw  [-,decorate,decoration=snake] (0,0) --(2,0);
----
\draw [->] (0,0) -- (-0.75,-1);
\node at (-0.7,-0.25) {$\alpha$};
\draw [-,decorate,decoration=snake] (-2.5, -1) -- (-0.5,-1);
----
\draw [->] (2,0) -- (2.75,-1);
\node at (2.6,-0.25) {$\beta$};
\draw [-,decorate,decoration=snake] (2.75,-1) -- (4.5,-1);
----
\node at (1,0.8) {$s$};
\node at (2,0.5) {$_{s(s)}$};
\node at (0,0.5) {$_{e(s)}$};
\node at (2,0.0) {$\bullet$};
\node at (0,0.0) {$\bullet$};
\end{tikzpicture}
\end{center}
\caption{General configuration of a quotient envelope of a string}
\label{fig_quotient factorization}
\end{figure}

\begin{remark} 
For $w \in \Str(A)$, and $s$ a substring of it, a quotient envelope ${\rm env}_w(s)$ is generally of the form depicted in Figure \ref{fig_quotient factorization}, where ${\rm env}_w(s)=\mu s \nu$ with $\mu = \alpha^{-1}$ and $\nu = \beta$, provided that $\mu$ and $\nu$ are non-trivial.
\end{remark}

 \begin{definition}
 Let $A$ be a gentle algebra, and let $w$ and $w'$ be two strings in $\Str(A)$ having respective substrings $s$ and $s'$. Assume that ${\rm env}_w(s) = \mu s \nu$ is a quotient envelope while ${\rm env}_{w'}(s') = \mu' s' \nu'$ is a submodule envelope. If $s'=s$ or $s'=s^{-1}$, and provided that all of $\mu, \nu, \mu', \nu'$ have positive length, we say there exists a \emph{kiss from $w$ to $w'$ along $s$}.
 \end{definition}

 For $w$ and $w'$ in $\Str(A)$, a kiss from $w$ to $w'$ is illustrated below, where the four arrows $\alpha$, $\beta$, $\gamma$, and $\delta$ must all appear and have the orientation depicted below. In this diagram, the string $w$ appears below and $w'$ appears above, and the substrings $s$ and $s'$ are identified, as illustrated by a thicker (common) substring of $w$ and $w'$:  

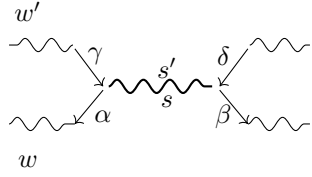
 \begin{center}
 \begin{tikzpicture}[scale=0.50]
 \draw  [-,decorate,decoration=snake][thick] (0.15,0) --(2.9,0);
 \node at (1.7,0.5) {$s'$};
 \node at (1.7,-0.4) {$s$};
 ------
 \draw [<-] (0,0) -- (-0.75,1);
 \draw [-,decorate,decoration=snake] (-0.8,1.1) -- (-2.5,1.1);
 \node at (-0.2,0.8) {$\gamma$};
 \node at (-2,2) {$w'$};
 -----
 \draw [<-] (3.1,0) -- (3.85,1);
 \draw [-,decorate,decoration=snake] (3.9,1.1) -- (5.5,1.1);
 \node at (3.2,0.8) {$\delta$};

 ------
 \draw [<-] (-0.75,-1) -- (0,-0.1);
 \draw [-,decorate,decoration=snake] (-2.5,-1) -- (-0.75,-1);
 \node at (0,-0.8) {$\alpha$};
 \node at (-2,-2) {$w$};
 ------
 \draw [<-] (3.85,-1) -- (3.1,-0.1);
 \draw [-,decorate,decoration=snake] (3.85,-1) -- (5.5,-1);
 \node at (3.2,-0.8) {$\beta$};
 \end{tikzpicture}
 \end{center}
 We emphasize that the notion of kiss is directed, and also a kiss from $w$ to $w'$ can be along a vertex, that is, the substrings $s$ and $s'$ in the above definition can be of length zero. We also remark that a string $w \in \Str(A)$ may admit a kiss to itself. That being the case, we say $w$ is a \emph{self-kissing} string.

\medskip
Let $w$ and $w'$ be two strings in $\Str(A)$, and assume that $s$ and $s'$ are respective substrings of $w$ and $w'$ where ${\rm env}_w(s)$ is a quotient envelope and ${\rm env}_{w'}(s')$ is a submodule envelope. Provided $s'=s^{-1}$ or $s'=s$, one can define a nonzero homomorphism from $M(w)$ to $M(w')$. More precisely, this homomorphism, henceforth denoted by $f_s$, is defined as the composition of the projection from $M(w)$ onto $M(s)$, followed by the identification of $M(s)$ with $M(s')$ or possibly $M((s'))^{-1})$, followed by the inclusion of $M(s')$ into $M(w')$. 
We refer to $f_s$ as a {\em graph map} from $M(w)$ to $M(w')$, determined by the substrings $s$ and $s'$, respectively in $w$ and $w'$. In particular, if $s$ and $s'$ are finite, we say $f_s$ is a \emph{finite} graph map. Obviously, the graph map $f_s$ in $\Hom_A(M(w),M(w'))$ depends on the chosen substrings $s,s'$, as well as the isomorphisms between the identification of $s'$ and $s$ (or $s'$ and $s^{-1}$). Hence, we use this notation only when the occurrences of $s$ in $w$ and of $s'$ in $w'$ are clear from the context.

\begin{remark} \label{automorphisms as graph maps} 

\begin{enumerate}
    \item A graph map $f_s: M(w) \to M(w)$ can be invertible only if $s, s' \in \{w, w^{-1}\}$. If $f_s$ is not the identity map, this occurs only when $w$ is double-infinite and periodic, and the graph map is given by identifying $w$ to one of its translation. Observe that we cannot have $w = w^{-1}$.

    \item For $w \in \Str(A)$, there is a non-trivial finite graph map $f_s:M(w)\to M(w)$ if and only if $w$ has a substring $s \ne w$ such that ${\rm Env}_w(s)$ contains both a submodule envelope and a quotient envelope. In particular, if $w$ is self-kissing, it admits such a non-trivial finite graph map, but the converse is not necessarily true; that is, a finite graph map in $\End_A(M(w))$ is not necessarily induced by a kiss.
\end{enumerate}
\end{remark}

As noted above, if $w \in \Str(A)$ is infinite, we can still associate to it a string module $M(w)$. If $\{x_i \mid i \in I\}$, for $I$ an interval of $\mathbb{Z}$, is the multiset of vertices visited by $w$, then the string module $M(w)$ can be constructed in the same way as if $w$ were finite (see Section \ref{Subsection: Strings}) and $\{x_i \mid i \in I\}$ indexes a basis of it. In particular, a basis of $M(w)$ at vertex $j$ is given by the multiset $\{x_i \mid i \in I, x_i = j\}$. See \cite{C-B2} for more detailed study of infinite-dimensional string modules. 
Moreover, we note that in the study of string modules, graph maps play an important role. In particular, we recall the following theorem. Observe that this result holds in a more general setting, but we only state for gentle algebras (for details, see \cite{C-B1}).

\begin{theorem}\label{Thm-graph-maps}
Let $A=kQ/I$ be a gentle algebra. For $u$ and $v$ in $\Str(A)$, the finite graph maps from $M(u)$ to $M(v)$ form a linearly independent set in $\Hom_A(M(u), M(v))$. Moreover, if $u$ and $v$ are finite strings, then the set of all graph maps is a basis for $\Hom_A(M(u), M(v))$.
\end{theorem}

\subsection{Bricks and inner bricks over gentle algebras}\label{Subsection: Bricks and inner bricks over gentle algebras}
Let $A$ be an algebra. A module $M$ in $\Modu A$ is called a \emph{brick} if $\End_A(M)$ is a division $k$-algebra, that is, every nonzero $A$-homomorphism $f:M\rightarrow M$ is invertible. If $M \in \modu A$, then a brick is sometimes called a \emph{Schur representation}, and this is the case if and only if $\End_A(M)\simeq k$, because $k$ is algebraically closed. For a recent survey on the study of bricks over finite dimensional algebras, we refer to \cite{MP2} and the references therein. Below, we are primarily interested in the study of bricks over gentle algebras.

If $A$ is a gentle algebra and $w \in \Str(A)$ is a finite string, Theorem \ref{Thm-graph-maps} implies that $M(w)$ is a brick if and only if $\End_A(M(w))$ contains no (non-trivial) graph map. In particular, if $w$ is a self-kissing string, $M(w)$ is not a brick. 
In this section, we generalize the above observations to arbitrary strings. We first define some weaker notions of bricks, which will be useful in our combinatorial arguments.

\begin{definition}
Let $A$ be a gentle algebra and $w \in \Str(A)$ be an arbitrary string.
\begin{enumerate}
 \item   The string $w$ is called an \emph{inner-brick} if it admits no self-kissing.
    \item  The string $w$ is called a \emph{strong inner-brick} if $\End_A(M(w))$ contains no non-trivial finite graph map.
\end{enumerate}

\end{definition}

\begin{remark} We make the following observations.
\begin{enumerate}
    \item Every strong inner-brick is obviously an inner-brick; simply because any self-kissing of $w$ induces a non-trivial finite graph map in $\End_A(M(w))$.
    \item If $w \in \Str(A)$ is not an inner-brick, there exists a $f\in \End_A(M(w))$ whose image is a (nonzero) finite dimensional indecomposable module, and both ${\rm ker}f$ and ${\rm coker}f$ are nonzero and decomposable. 
    \item If $w$ is double-infinite, the notions of inner-brick and strong inner-brick coincide. We will see later of examples of inner-bricks that are not strong inner-bricks.
    \item A strong inner-brick $w \in \Str(A)$ needs not be a brick, because we may have a non-zero non-invertible $f\in \End_A(M(w))$ that has infinite-dimensional image. 
    For instance, for the Kronecker quiver $Q:1\doublerightarrow{\alpha}{\beta}2$, and the gentle algebra $A=kQ$, consider $w:=\alpha\beta^{-1}\alpha\beta^{-1} \cdots$. Obviously, $w$ is a right-periodic, and an easy computation shows that there is no finite graph map $M(w) \to M(w)$, so $M(w)$ is a strong inner-brick. The factorization $w = \alpha \beta^{-1}w$ shows that $w$ as the whole string can be considered as a submodule substring because its equals its own envelope, and the right end $w$ of the above factorization can be considered as a quotient substring of $w$, because ${\rm env}_w(w) = \beta^{-1}w$. Hence, there is a graph map $M(w) \to M(w)$ that is not invertible, so $M(w)$ is not a brick.
\end{enumerate}
\end{remark}

\subsection{Covering techniques}

To study infinite-dimensional string modules over a gentle algebra, it will be convenient to use some methods coming from covering theory. Here we present some necessary tools, but we restrict our exposition to the case of monomial algebras. For details on covering theory in the context of representation theory of algebras, see \cite{BG} and references therein.

\medskip

Let us fix $A$ a connected monomial algebra given by the bound quiver $(Q,I)$, that is, $A=kQ/I$, and the admissible ideal $I$ is generated by a set of paths of length at least two in $Q$.
We fix $x_0$ a vertex of $Q$ and let $G:=\pi_1(Q,x_0)$ be the set of all reduced cyclic walks in $Q$ from $x_0$ to $x_0$. We note that $G$ admits the structure of a finitely generated free group, called the \emph{fundamental group} of the underlying graph of $Q$. More precisely, the product in $G$ is given by concatenation, and using that for an arrow $\alpha: i \to j$, we have $\alpha \alpha^{-1} = \epsilon_i$ and $\alpha^{-1}\alpha = \epsilon_j$. We note that the structure of $G$ is independent of the chosen vertex $x_0$ and of the ideal $I$. Moreover, the number of free generators of $G$ is $|Q_1| - |Q_0| + 1$. 

We construct the universal Galois covering of $(Q,I)$ as follows. Let $\Gamma$ be the infinite quiver given by a tree whose vertices are given by the finite reduced walks in $Q$ starting at $x_0$. If $w', w$ are two reduced walks starting at $x_0$ with $w = w'\alpha$ where $\alpha: i \to j$ is an arrow, then we put an arrow $w' \to w$ in $\Gamma$. The group $G = \pi_1(Q,x_0)$ acts on $\Gamma$ by left multiplication. It acts freely on the vertices of $\Gamma$. We also note that the orbit quiver $\Gamma / G$ of $\Gamma$ by $G$ can be identified with $Q$. Let $F: \Gamma \to Q$ be the quiver map that sends any vertex or arrow to its $G$-orbit. Clearly, $F$ can be extended to a $k$-linear map $F: k\Gamma \to kQ$ of path algebras. We let $J$ be the monomial ideal of $\Gamma$ of relations that are sent to $I$. The induced map $F: k\Gamma/J \to kQ/I$ is the \emph{universal Galois cover} of $A = kQ/I$. Set $B: = k\Gamma /J$.

For algebra $B$ as above, let $\Modu B$ (respectively $\modu B$) denote the category of all (respectively of locally finite dimensional) representations of $B$. Observe that $G$ acts on $\Modu B$ naturally. For $g \in G$ and $M \in \Modu B$, let $M^g$ denote the $B$-module obtained via acting on $M$ by $g$. The functor $F$ given above induces a functor $F_\lambda: \Modu B \to \Modu A$, called the \emph{push-down} functor. 
To describe this functor on the representations, let $M \in \Modu B$. Then, for $x \in Q_0$, we have $F_\lambda(M)(x) = \bigoplus_{Fx' = x}M(x')$, where the direct sum runs over the preimage by $F$ of $x$ in $\Gamma_0$. Moreover, if $\alpha: x \to y$, then the linear transformation $F_\lambda(M)(\alpha) : \bigoplus_{Fx' = x}M(x') \to \bigoplus_{Fy' = y}M(y')$ is simply the diagonal ${\rm diag}(M(\alpha') \mid F\alpha' = \alpha)$. 

Note that the push-down functor preserves the total dimension of a representation. Moreover, this functor admits a right adjoint, known as the \emph{pull-up} functor $F_\bullet: \Modu A \to \Modu B$, which is induced from $F: B \to A$ so that if $M$ is an $A$-module, then $F_\bullet M = M \circ F$ where $M$ is thought of as a $k$-linear functor $M: A \to \Modu k$.  
An important property is the following. If $M \in \Modu B$, then $F_\bullet F_\lambda M = \bigoplus_{g \in G} M^g.$ Using the adjoint property, for $M \in \Modu B$, we have that 
\[
\Hom_A(F_\lambda M, F_\lambda M) \cong \Hom_B(M, F_\bullet F_\lambda M) \cong \Hom_B(M, \oplus_{g \in G}M^g). 
\]
Using this, we get the following useful fact.
\begin{prop}
    Let $M \in \Modu B$ with $\End_B(M)=k$. If $\Hom_B(M, M^g)=0$ for all $1 \ne g \in G$, then $\End_A(F_\lambda M) = k$.
\end{prop}

\begin{proof}
    As shown above, we have 
    \[
\Hom_A(F_\lambda M, F_\lambda M) \cong \Hom_B(M, \oplus_{g \in G}M^g).
\]
Let $f: M \to \oplus_{g \in G}M^g$ be a non-zero morphism. For $1 \ne h \in G$, the projection $p_h:\oplus_{g \in G}M^g \to M^h$ gives $p_hf=0$ by assumption. This shows that the image of $f$ is entirely contained in the summand $M^1$. Hence, $\End_A(F_\lambda M) \cong \End_B(M)=k$.
\end{proof}

This result will be useful for studying infinite-dimensional string modules over a gentle algebra $A$. 

 \begin{prop} \label{infinite-dimensional bricks and graph maps}
     Let $A$ be a gentle algebra and $w\in \Str(A)$. The string module $M(w)$ is a brick if and only if the only  graph map from $M(w)$ to itself is the identity.
 \end{prop}

 \begin{proof}
     For $w\in\Str(A)$, the string module $M(w)$ in $\Modu A$ lies in the essential image of the push-down functor $F_\lambda$. Indeed, the string $w$ can be lifted to a string $\hat{w}$ in $(\Gamma, J)$. From the construction of a string module, it follows that $F_\lambda M(\hat{w}) \cong M(w)$. 
     Moreover, $M(\hat{w})$ is such that its dimension at any given vertex of $\Gamma$ is zero or one; such $M(\hat{w})$ is often called a thin module. Hence, we evidently have $\End_B(M(\hat{w})) = k$.  Now, we observe that $\Hom_B(M(\hat{w}), M(\hat{w})^g) \ne 0$ if and only if there is a graph map from $M(\hat{w})$ to $M(\hat{w}^g)$. This gives that $\End_A(M(w))=k$ if and only if for each $g \ne 1$, we have $\Hom(M(\hat{w}), M(\hat{w})^g)=0$. This is equivalent to stating that there is no graph map, other than the identity graph map, from $M(w)$ to $M(w)$. 
     Finally, recall from Remark \ref{automorphisms as graph maps} that if we have a non-identity graph map that is invertible, then $w$ is double-infinite and periodic. In this case, we have that $\End(M(w)) = k[x, x^{-1}]$ and in particular, $M(w)$ is not a brick.
 \end{proof}

\medskip

\section{Double-Kronecker algebra}
\label{section: double-Kronecker algebra}
In this section, we consider the \emph{double-Kronecker} quiver with a fixed labeling of the arrows, given as follows
\[Q = \xymatrix{1 \ar@/^/[r]^{\alpha_1}\ar@/_/[r]_{\alpha_2}& 2 \ar@/^/[r]^{\beta_1}\ar@/_/_{\beta_2}[r] & 3}\]
Observe that, up to relabeling the arrows of $Q$, there is a unique ideal $I \subseteq \rad^2(kQ)$ such that $kQ/I$ is a gentle algebra. Throughout this section, unless specified otherwise, we let $A = kQ/I$ denote the gentle algebra induced by the double-Kronecker quiver $Q$, where $I = \langle \alpha_1\beta_1, \alpha_2\beta_2 \rangle.$
Moreover, we consider two cyclic strings in $\Str(A)$, given by $a = \alpha_1^{-1}\alpha_2$ and $b = \beta_1\beta_2^{-1}$. 

\medskip

Let $\Str(a,b)$ denote the subset of $\Str(A)$ consisting of all strings that can be written as a word in $a$ and $b$. More specifically, each $w \in \Str(a,b)$ is a binary word in $a$ and $b$ (and no copies of $a^{-1}$ and $b^{-1}$). 
We note that every finite string in $\Str(a,b)$ starts and ends at $2$. In particular, the trivial string at vertex $2$ is considered to be an element of $\Str(a,b)$.

\begin{remark} With the same notation as above, we make the following remarks.
\begin{enumerate}
    \item In this section, we will not distinguish between elements of $\Str(a,b)$ and binary words in $a,b$. On the other hand, we will also need to consider some strings that are not necessarily in $\Str(a,b)$. Observe that an arbitrary substring of an element of $\Str(a,b)$ is not necessarily in $\Str(a,b)$.
    
    \item Let $w \in \Str(a,b)$. If $x$ is a substring of $w$ of positive length, note that $x^{-1}$ is not a substring of $w$.
\end{enumerate}
\end{remark}

We now define the analogue of the notion of envelope from Definition \ref{Def: general envelope} but using the alphabet $\{a,b\}$ specified above. In particular, for $w,s \in \Str(a,b)$, if $s$ is a substring of $w$, the following notion of $a,b$-envelope of $s$ in $w$ produces a (possibly longer) substring of $w$ that contains $s$ and again belongs to $\Str(a,b)$.

\begin{definition}\label{Def: a-b envelope}
    Let $s, w \in \Str(a,b)$ with $s$ a substring of $w$.
    \begin{enumerate}
        \item Given an occurrence of $s$ as a substring of $w$, we define the \emph{$a,b$-envelope} ${\rm env}_w^{a,b}(s)$ of $s$ in $w$ to be the unique substring $s'$ of $w$ containing $s$ with the property that $s'$ is obtained from $s$ by appending one letter from $\{a,b\}$ to the left of $s$ whenever $s$ is not a left end of $w$; and appending one letter from $\{a,b\}$ to the right of $s$ whenever $s$ is not a right end of $w$.

        \item We define ${\rm Env}_w^{a,b}(s)$ to be the set of all $a,b$-envelopes ${\rm env}_w^{a,b}(s')$ where $s'$ runs through all possible occurrences of $s$ as a substring in $w$.
    \end{enumerate}
\end{definition}

\medskip

The next lemma follows from Lemma \ref{Lemma envelopes}.

\begin{lemma}
    \label{Strong Inner Brick}
    Let $w$ be a string in $\Str(a,b)$. Then $M(w)$ is a strong inner brick if and only if there is no finite $x \in \Str(a,b)$ such that ${\rm Env}_w^{a,b}(x) \cap \{axa, ax, xa\} \ne \emptyset$ and ${\rm Env}_w^{a,b}(x) \cap \{bxb, bx, xb\} \ne \emptyset$.

\end{lemma}

\begin{proof}
Given an arbitrary finite substring $x$ of $w$ with $x \ne w$, we may list all possible quotient envelopes of $x$ in $w$ and all possible submodule envelopes of $x$ in $w$.

\medskip
\noindent 
\begin{minipage}[t]{0.48\textwidth} 
    \centering 
    \begin{tabular}{|c|c|c|}
    \hline
    Quotient $ {\rm env}_w(x)$ & $s(x)$ & $t(x)$\\
    \hline
    $\alpha_1^{-1} x \alpha_2$ & 1 & 1 \\
    $\alpha_1^{-1} x \beta_1 $ & 1 & 2 \\
    $\beta_2^{-1} x \alpha_2 $ & 2 & 1 \\
    $\beta_2^{-1} x \beta_1 $ & 2 & 2 \\
    $ x \alpha_2$ & 2 & 1 \\
    $ x \beta_1 $ & 2 & 2 \\
    $\alpha_1^{-1} x $ & 1 & 2 \\
    $\beta_2^{-1} x $ & 2 & 2 \\
    \hline
    \end{tabular}
\end{minipage}
\hfill 
\begin{minipage}[t]{0.48\textwidth}
    \centering
    \begin{tabular}{|c|c|c|}
    \hline
    Submodule ${\rm env}_w(x)$ & $s(x)$ & $t(x)$\\
    \hline
    $\alpha_2 x \alpha_1^{-1} $ & 2 & 2 \\
    $\beta_1 x \alpha_1^{-1}$ & 3 & 2 \\
    $\alpha_2 x \beta_2^{-1}$ & 2 & 3 \\
    $\beta_1 x \beta_2^{-1} $ & 3 & 3 \\
    $ x \alpha_1^{-1}$ & 2 & 2 \\
    $ x \beta_2^{-1} $ & 2 & 3 \\
    $\alpha_2 x $ & 2 & 2 \\
    $\beta_1 x $ & 3 & 2 \\
    \hline
    \end{tabular}
\end{minipage}

\medskip

Observe that among the pairs $(s(x), t(x))$ from the above tables, only $(2,2)$ appears in both tables. Therefore, the only substrings of $w$ that can have both submodule and quotient envelopes must start and end at vertex 2. Hence, such a finite substring $x$ has to be in $\Str(a,b)$. Observe now that it then follows from the construction of $a$ and $b$ that ${\rm env}_w(x)$ uniquely determines ${\rm env}^{a,b}_w(x)$. The following table shows this correspondence, where the quotient envelopes are shown first, followed by the submodule envelopes.

\medskip

{
\centering
    \begin{tabular}{|c|c|}
    \hline
    ${\rm env}_w(x)$ & ${\rm env}^{a,b}_w(x)$\\
    \hline
    $\beta_2^{-1} x \beta_1 $ & $bxb$  \\
    $ x \beta_1 $ & $xb$  \\
    $\beta_2^{-1} x $ & $bx$  \\
     &   \\
    $\alpha_2 x \alpha_1^{-1} $ & axa  \\
    $x \alpha_1^{-1} $ & xa  \\
    $\alpha_2 x $ & ax  \\
    \hline
    \end{tabular}
    
    }

\medskip
Therefore, there exists no such $x$ with ${\rm Env}_w^{a,b}(x)$ including one of $\{bxb,xb,bx\}$ and one of $\{axa,xa,ax\}$ if and only if $w$ is a strong inner brick. 
\end{proof}

\begin{remark}\label{Inner Brick}
Note that when considering inner-bricks, those submodule substrings and also the quotient substrings which are at the extremities of $w$ are not considered. Thus, $M(w)$ is an inner brick if and only if there is no finite $x \in \Str(a,b)$ such that $axa \in {\rm Env}_w^{a,b}(x)$ and $bxb \in {\rm Env}_w^{a,b}(x)$. 
\end{remark}

The following propositions establish useful connections between some strings of the double-Kronecker quiver and some Sturmian words introduced in Section \ref{Section: binary words}.

\begin{prop}
    \label{Inner Brick Sturmian}
    Let $w$ be an aperiodic infinite string in $\Str(a,b)$. Then $M(w)$ is an inner brick if and only if it is a Sturmian word in $a$'s and $b$'s.
\end{prop}

\begin{proof}
    This follows from Remark \ref{Inner Brick} and Proposition \ref{Prop: For aperiodic, Sturmian is axa, bxb avoiding}.
\end{proof}

\begin{prop} \label{Lemma sturmian inner-bricks}

    An aperiodic double-infinite string in $\Str(A)$ is a strong inner brick if and only if it is a Sturmian word in $\Str(a,b)$.
\end{prop}

    \begin{proof} For a double-infinite string, being a strong-inner brick is equivalent to being an inner brick, therefore this follows from Proposition \ref{Inner Brick Sturmian}
    \end{proof}

    \begin{lemma}
    A right-infinite aperiodic string in $\Str(a,b)$ is a strong inner brick if and only if it is a characteristic Sturmian word.
    \label{le:characteristicbrick}
    
    \begin{proof}
    A right-infinite string $s$ in $\Str(a,b)$ will be a strong inner brick if and only if $s$, $as$, and $bs$ are inner bricks. Therefore, $s$ is a strong inner brick if and only if $s$, $as$, and $bs$ are Sturmian words. By Proposition \ref{sturmian with lattice point}, this is equivalent to $s$ being a characteristic Sturmian word. 
    \end{proof}
\end{lemma}

    In the next lemma, we give the necessary and sufficient condition for some infinite elements of $\Str(a,b)$ that induce bricks.

\begin{lemma}
    Let $w \in \Str(a,b)$ be infinite and let $M(w)$ be a strong inner-brick. Then $M(w)$ is a brick if and only if $w$ is aperiodic.

    \label{le:aperiodicbrick}
\end{lemma}

\begin{proof}
Let $M(w)$ be a strong inner brick, meaning that there are no finite dimensional graph maps between $M(w)$ and itself. It follows from Proposition \ref{infinite-dimensional bricks and graph maps} that $M(w)$ is a brick if and only if there is no infinite-dimensional graph maps between $M(w)$ and itself, other than the identity. 
For an infinite substring $s$ of $w$, note that $s^{-1}$ cannot be a substring of $w$. Hence, the existence of an infinite-dimensional graph map different from the identity is equivalent to the existence of an infinite substring $s$ of $w$ with at least two occurrences within $w$. In all cases, this is equivalent to $w$ being eventually periodic or periodic. If $w$ is  periodic, then we have already seen that $M(w)$ is not a brick. So, we let $w$ be eventually periodic. 

Without loss of generality we can assume it is eventually right periodic. In particular, suppose $w = w'ssss\dots$. 
Note that this way of writing $w$ is not unique, and we may choose $w', s \in \Str(a,b)$ such that the last letter of $w'$ is different from the last letter of $s$. The only scenario in which this is not possible is when $w= \dots sss \dots$, which was treated above as the periodic case. Excluding that scenario, we have the infinite string $sss\dots$ preceded by an $a$ in one instance and by $b$ in another, since $w'$ ends with one of them and $s$ with the other. This results in a valid graph map that is not invertible, and implies that $M(w)$ is not a brick.
\end{proof}

\begin{prop}
Let $w\in \Str(a,b)$ be right-infinite and consider the string $\alpha_2 w$. Then $M(\alpha_2 w)$ is a brick if and only if there exists a characteristic Sturmian word $w'$ such that $w=bw'$. 
\end{prop}

\begin{proof}

We claim that $M(\alpha_2 w)$ is a brick if and only if $aw$ is Sturmian, and for any finite string $s$ with $sa$ an initial substring of $w$, any other occurrence of $s$ in $w$ is followed by an $a$. 

First, assume that $M(\alpha_2 w)$ is a brick. From Remark \ref{Inner Brick}, it follows  that $aw$ is a Sturmian word. 
   Observe that if $sa$ is a starting substring of $w$, then the initial substring $\alpha_2 s$ of $\alpha_2 w$ provides a submodule substring of $\alpha_2 w$. Assume that $sb$ occurs in $w$. If $sb$ is preceded by an $a$, then we get a quotient substring $\alpha_2 s$. If $sb$ is preceded by a copy of $b$, then we get a contradiction to $aw$ being Sturmian.  
    
    Conversely, we check that if $M(\alpha_2 w)$ is not a brick, but $aw$ is Sturmian, then there is a finite substring $s$ such that $sa$ is an initial substring of $w$ and there is an occurrence of $sb$ in $w$. Since $\alpha_2 w$ is right-infinite and non-periodic (because $aw$ is Sturmian), there is no graph map $M(\alpha_2 w) \to M(\alpha_2 w)$ with infinite-dimensional image. Let $s_1$ be a finite quotient substring and $s_2$ be a finite submodule substring of $M(\alpha_2 w)$ with $s_1 = s_2$. Note that the only way for a substring $\alpha_2 s$ to be a submodule substring is when it is an initial substring of $\alpha_2 w$. Hence, the case where $s_1$ is an initial substring of $\alpha_2 w$ is excluded. Hence, $s_1$ is a quotient substring of $w$ (and observe that it cannot use the first arrow of $w$). If $s_2$ is an initial substring of $\alpha_2 w$, then $s_2 = \alpha_2 s$ that is followed by an $a$ and $s_1$ has to be $s_1 = \alpha_2 s$ which is followed by a $b$. Hence, we get the desired result. The remaining case is when  $s_1, s_2$ are entirely contained in $w$. As we observed above, $s_1$ does not use the first arrow of $w$. Hence, $aw$ is such that there is $s$ with $asa, bsb$ in it, contradicting that $aw$ is Sturmian. This finishes the proof of the claim.

Now, let us assume that $M(\alpha_2 w)$ is a brick, so that $aw$ is Sturmian and satisfies the other property of the above claim. We know that $aw$ is given by a cutting word $r^D_{m,c}$, for some irrational positive number $m$ and some right-infinite interval $D$. If $y=mx+c$ goes through a (single) lattice point over $D$, then $aw$ is either the lower cutting sequence which associates to this lattice point a copy of $ba$, or it is the upper cutting sequence which associates to this lattice point an $ab$. We will have to distinguish between these two possible cases later. Let us consider the set
$$\mathcal{S} = \{s \mid sab \; \text{is an initial subword of}\; w\}.$$
 Such an $s$ represents an initial segment of the cutting word such that the last two cuts are horizontal and then vertical (or a cut that is a lattice point, in case we are in the upper cutting sequence case). Every such vertical cut corresponds to an integral value on the $x$-axis.
We take $D \cap \mathbb{Z} = [x_1, \infty)$ and we write $x_i = x_1 + i-1$, for $i \in \mathbb{Z}_{>0}$. For each $x_i$, we compute $z_i:= y(x_i) - \lfloor(y(x_i))\rfloor$, with $0 < z_i < 1$ except possibly for one value of $z_i$, provided the line $y=mx + c$ goes through a lattice point over $D$.

The set $\{z_1, z_2, \ldots\}$ is dense in the interval $[0,1]$. Hence, there are infinitely many $j$ such that $z_j$ is the minimum of the set $\{z_1, z_2, \ldots, z_j\} \setminus \{0\}$. Such a $j$ necessarily corresponds to a vertical crossing where the previous crossing was horizontal (if not, the previous vertical crossing would have a $z$ value strictly smaller). Hence, such a $j$ gives rise to an initial subword of $w$ of the form $sab$. 

Consider first the case where $y=mx+c$ does not go through any lattice point over $D$. If we take a $j$ as above, we see that over $D \cap (-\infty, x_j]$, translating the line $y = mx+c$ down by a value $\epsilon$ with $0 < \epsilon < z_j$ does not change the cutting sequence over  $D \cap (-\infty, x_j]$. In fact, there exists some $\epsilon_j > z_j$ such that   translating the line $y = mx+c$ down by a value $\epsilon_j$ does not change the cutting sequence over  $D \cap (-\infty, x_j]$, except that the last cuts ``horizontal, then vertical" become ``vertical". This means that we have a starting substring of the form $sab$ for $w$, and it is possible to translate the line $y=mx+c$ down by some small value so that the new line has a cutting word with initial subword given by $sb$. It follows from the irrationality of $m$ that the cutting sequence corresponding  to $sb$ will eventually occur in $w$ over $D$. But this implies that $M(\alpha_2 w)$ is not a brick, hence a contradiction.

In the next case, assume that $y=mx+c$ goes through a lattice point over $D$ and that we are in the lower cutting sequence case (so that this lattice point corresponds to a $ba$). Then the above argument still applies, because translating the line down by a small value will turn the lattice point into two crossings, "vertical, then horizontal", which is also represented by $ba$. 

The very last case is when $y=mx+c$ goes through a lattice point over $D$ and that we are in the upper cutting sequence case. In this case, if the lattice point is not the first time the line $y=mx+c$ cuts the integer grid over $D$ we have $aw = az(ab)z^Ta w'$ (due to the symmetry around the lattice crossing)
where $z$ is some (possibly empty word). It follows from Proposition \ref{Both sa and sb} that since $z^Ta$ is the start of the characteristic Sturmian word $z^Taw''$, this means that $(z^Ta)^T = az$ has the property that both $aza, azb$ appear in $aw$. As show above, this means that $M(\alpha_2 w)$ is not a brick. 

We are now left with the case $aw = (ab)w'$, where the first factor $ab$ corresponds to the lattice point, and $w'$ is the characteristic Sturmian word of slope $m$. Assume that $aw$ has a starting subword $asa$ and that $asb$ also appears in $aw$. Let us denote by $d$ the letter such that $asad$ is a starting subword of $aw$. This means that $(as)^T = s^Ta$ is a starting subword of $w'$. However, $s$ starts with $b$ and the equality $asad = abs^Ta$ implies that $s$ also starts with $a$, a contradiction. Hence, this is the only case that leads to a brick $M(\alpha_2 w)$. Since $aw = (ab)w'$ is such that $w'$ is a characteristic Sturmian word, the proof is complete.
\end{proof}

The following theorem follows from our analysis above and it gives a complete description of infinite string bricks over the $2$-Kronecker algebra. Recall that, for $w \in \Str(A)$, by $s(w)$ we denote the start of $w$.

\begin{theorem}\label{Thm: sturmian words of double-Kronecker}
    With the same notation as above, let $A=kQ/I$ be the gentle algebra of the double-Kronecker quiver $Q$. For an infinite string $w$ over $A$, we have 
    \begin{enumerate}
        \item If $s(w)=1$, then $M(w)$ is a brick if and only if either $w = \alpha_2 b w'$ with $w'$ a characteristic Sturmian word in $\Str(a,b)$ or else $w = \alpha_1 b^{-1} w'$ with $w'$ a characteristic Sturmian word in $\Str(a^{-1},b^{-1})$.
        
        \item If $s(w)=2$, then $M(w)$ is a brick if and only if $w$ is a characteristic Sturmian word in $\Str(a,b)$ or a characteristic Sturmian word in $\Str(a^{-1},b^{-1})$.
       
        \item If $s(w)=3$, then $M(w)$ is a brick if and only if either $w = \beta_2^{-1} a w'$ with $w'$ a characteristic Sturmian word in $\Str(a,b)$ or else $w = \beta_1^{-1} a^{-1} w'$ with $w'$ a characteristic Sturmian word in $\Str(a^{-1},b^{-1})$. 
        
        \item If $w$ is left-infinite, then $M(w)$ is a brick if and only if $M(w^{-1})$ is a brick. 
        
        \item If $w$ is double-infinite, then $M(w)$ is a brick if and only if $w$ or $w^{-1}$ is a Sturmian word in $\Str(a,b)$.
    \end{enumerate}
 \end{theorem}

 \subsection*{Some related combinatorics of words over gentle algebras}\label{Subsect: Some related combinatorics of words over gentle algebras}
 In this short subsection, we highlight some other recent studies on the linkages between combinatorics of (binary) words and the representation theory of gentle algebras. In particular, we compare our characterization of Sturmian words (see Theorem \ref{Thm: sturmian words of double-Kronecker}) with the main result of \cite{DL+}, where the authors give a new characterization of the perfectly clustering words in terms of some bricks over a particular family of gentle algebras (see also \cite{KS}). The authors study the $n$-fold Kronecker algebra $A_n$ with quiver given by Figure \ref{Fig: n-fold Kronecker}.

  \begin{figure}[h]
    \centering
    \[\begin{tikzcd} Q_{n}: &
 1 \arrow[r, , bend left] \arrow[r, , bend right] & 2 \arrow[r, , bend left] \arrow[r, , bend right] & \cdots \arrow[r, bend right] \arrow[r, bend left] & {n-1} \arrow[r, , bend left] \arrow[r, , bend right] & n & 
\end{tikzcd}\]
    \caption{$n$-fold Kronecker gentle algebra.}
    \label{Fig: n-fold Kronecker}
\end{figure}
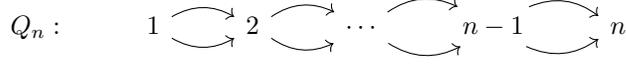
 
They naturally associate to $A_n$ some brick bands $z_1, \ldots, z_{n-1}$, all starting and ending at vertex $1$, that can be composed freely, and they consider some words in the $z_i$ called perfectly clustering words. They show that the perfectly clustering words in the $z_i$ are in bijection with the brick bands that start and end at vertex $1$. It turns out that when $n=3$, which is the double-Kronecker gentle algebra, their theorem, in our language, can be stated as follows
\begin{theorem}[Special case of \cite{DL+}]\label{Thm: perfectly clustering words}
    The brick bands starting and ending at vertex $2$ over the double-Kronecker gentle algebra are exactly the words $w$ in $\Str(a,b)$ where $bwa$ is a Christoffel word.
\end{theorem}

As observed in Proposition \ref{link Sturmian and Christoffel}, (characteristic) Sturmian words can be seen as limits of Christoffel words. Hence, we can interpret our infinite string bricks as some limits of finite brick bands. We also observe that truncating a characteristic Sturmian word in an arbitrary way does not, in general, give a Christoffel word.

\section{Single-kissing brick bands}

In this section, we generalize our results from the preceding sections. In particular, to extend our setting from the double-Kronecker quiver treated in Section \ref{section: double-Kronecker algebra}, we consider a gentle algebra $A = kQ/I$ and let $a = za',b = zb'$ be two strings in $\Str(A)$ which give rise to brick-bands. We further assume that $a$ and $b$ start (and end) at vertex $x$, and they do not pass through $x$ in between. Moreover, assume that $ab$ and $ba$ are strings. This particularly implies that any binary word in letters $a$ and $b$ belongs to $\Str(A)$. Furthermore, we assume that $z$ gives rise to a kiss from $b^2$ to $a^2$ and that this is the only kiss from a rotation of $b$ to a rotation of $a$.
Without loss of generality, we can therefore assume that $a'$ starts with the inverse of an arrow, say $\alpha_1^{-1}$, and ends with a direct arrow, say $\alpha_2$, while $b'$ starts with a direct arrow, say $\beta_1$, and ends with an inverse arrow, say $\beta_2^{-1}$; see Figure \ref{Kiss Section 4}.

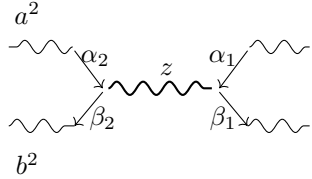
\begin{figure}[h]
    \centering
    \begin{tikzpicture}[scale=0.50] 
\draw  [-,decorate,decoration=snake][thick] (0.15,0) --(2.9,0);
\node at (1.7,0.5) {$z$};
------
\draw [<-] (0,0) -- (-0.75,1);
\draw [-,decorate,decoration=snake] (-0.8,1.1) -- (-2.5,1.1);
\node at (-0.2,0.8) {$\alpha_2$};
\node at (-2,2) {$a^2$};
-----
\draw [<-] (3.1,0) -- (3.85,1);
\draw [-,decorate,decoration=snake] (3.9,1.1) -- (5.5,1.1);
\node at (3.2,0.8) {$\alpha_1$};

------
\draw [<-] (-0.75,-1) -- (0,-0.1);
\draw [-,decorate,decoration=snake] (-2.5,-1) -- (-0.75,-1);
\node at (0,-0.8) {$\beta_2$};
\node at (-2,-2) {$b^2$};
------
\draw [<-] (3.85,-1) -- (3.1,-0.1);
\draw [-,decorate,decoration=snake] (3.85,-1) -- (5.5,-1);
\node at (3.2,-0.8) {$\beta_1$};
\end{tikzpicture}
    \caption{Schematic presentation of a kiss from $b^2$ to $a^2$ (or from a rotation of $b$ to a rotation of $a$)}
    \label{Kiss Section 4}
\end{figure}

\begin{remark} From our assumption on $x$, it follows that
    \begin{enumerate}
        \item $\alpha_2\beta_2 \in I$  and $\alpha_1\beta_1 \in I$.
        \item  $\alpha_1, \alpha_2, \beta_1, \beta_2$ are pairwise distinct.  Moreover, these $4$ arrows do not appear anywhere else in $a,b$ other than the places shown on Figure \ref{Kiss Section 4}.
        \end{enumerate}
\end{remark}

\begin{definition}
Let $s$ be a finite string over $A$. By $x(s)$, called the \emph{$x$-count of $s$}, we denote the number of times that vertex $x$ is visited by the walk induced by $s$.  
\end{definition}

In the sequel, we will make use of any terminology of the previous sections, when it applies. For instance, we use the notion of envelope of an occurrence of a substring inside a string, and the corresponding notions of submodule substring (or envelope) and quotient substring (or envelope). In order to use the tools from the previous sections, the following lemma is crucial.

\begin{lemma} \label{Lemma: followed by w}
    Let $w \in \Str(a,b)$ and let $s$ be a finite string. If $s$ appears both as a submodule and a quotient substring in $w$,  then $s= s'z$ where $s' \in \Str(a,b)$. \label{longer kisses}
\end{lemma}
\begin{proof}
    Suppose $s$ appears both as a submodule and a quotient substring in $w$.
    Note that there is a unique string $s'$ in $\Str(a,b)$ with $s = p_1s'p_2$ for which we have $x(s') = x(s)$, $x(p_1) = x(p_2) = 1$ and $t(p_1) = x = s(p_2)$. 
    Observe that these properties imply that $p_1$ is either a proper ending substring of $a$ or of $b$. In other words, $p_1$ satisfies at least one of the following properties: 

    \begin{enumerate}
        \item $a = up_1$
        \item $b = vp_1$
    \end{enumerate}
    where $u,v$ have positive length. Assume that $p_1$ has positive length. Then only one of those can occur, say case (1). If $u$ ends with an arrow, then $s$ cannot appear in a quotient substring of $w$, which is a contradiction. Similarly, if $u$ ends with the inverse of an arrow, then $s$ cannot appear in a submodule substring of $w$, which is also a contradiction. Therefore, we conclude that $p_1 = e_x$ has length zero.
 
   We argue a similar thing for the end of $s$, namely

    \begin{enumerate}
        \item[$(3)$] $a = p_2u'$
        \item[$(4)$] $b = p_2v'$
    \end{enumerate}

   We claim that $p_2 = z$. For this, it suffices to prove that $p_2$ and $z$ have the same length. Suppose first that the length of $p_2$ is less than the length of $z$. In this case, $u',v'$ starts with the same letter. This contradicts the fact that $s$ appears in both a submodule and a quotient substring in $w$. Consider finally the case where the length of $p_2$ is greater than the length of $z$. In this case, only one of (3), (4), say case (3), can hold since $\alpha_1 \ne \beta_1$. If $u'$ starts with an arrow, then $s$ cannot appear in a submodule substring of $w$, which is a contradiction. Similarly, if $u'$ starts with the inverse of an arrow, then $s$ cannot appear in a quotient substring of $w$, which is also a contradiction. Therefore, we conclude that $p_2 = z$. This shows that $s = p_1 s' p_2 = s'z$ where $s' \in \Str(a,b)$.
    \end{proof}

    We now present the analogue of Lemma \ref{Strong Inner Brick} and Remark \ref{Inner Brick} from previous section.

\begin{lemma}
    \label{section4stronginner}
     Let $w$ be a string in $\Str(a,b)$. Then $M(w)$ is a strong inner brick if and only if there is no finite $s \in \Str(a,b), s \ne w$, such that ${\rm Env}_w^{a,b}(s) \cap \{asa, as, sa\} \ne \emptyset$ and ${\rm Env}_w^{a,b}(s) \cap \{bsb, bs, sb\} \ne \emptyset$.
\end{lemma}

\begin{proof}
We know that $M(w)$ is not a strong inner brick if and only if there is a finite string $s$ such that $s$ appears both as a submodule and a quotient substring of $w$. It follows from Lemma \ref{Lemma: followed by w} that $s = w'z$ where $w' \in \Str(a,b)$. Let $s_1$ denote an occurrence of $s$ in $w$ corresponding to a quotient envelope. Then ${\rm env}_w(s_1) = \mu_1 s_1 \mu_2 = \mu_1 w' z \mu_2$ where $\mu_1$ is trivial or $\mu_1 = \beta_2^{-1}$, and $\mu_2$ is trivial or $\mu_2 = \beta_1$. We note that the case where both $\mu_1, \mu_2$ are trivial does not happen since $s_1 \ne w$. These cases lead to ${\rm env}^{a,b}_w(s_1) \in \{bw', w'b,bw'b\}$. Similarly, considering $s_2$ an occurrence of $s$ as a submodule envelope, we get that ${\rm env}^{a,b}_w(s_2) \in \{aw', w'a,aw'a\}$.

Conversely, it is clear that if ${\rm Env}_w^{a,b}(s) \cap \{asa, as, sa\}$ is non-empty, then there is an occurrence of $sz$ that gives a submodule substring of $w$. Similarly, if ${\rm Env}_w^{a,b}(s) \cap \{bsb, bs, sb\}$ is not empty,  then there is an occurrence of $sz$ that gives a quotient substring of $w$.
\end{proof}

\begin{remark}    
\label{section4inner}
 Considering the similar notion of inner-brick, for which we discard submodule and quotient substrings at the extremities of $w$, we see that $M(w)$ is an inner brick if and only if there is no finite $s \in \Str(a,b)$ such that ${\rm Env}_w^{a,b}(s) \cap \{asa\} \ne \emptyset$ and ${\rm Env}_w^{a,b}(s) \cap \{bsb\} \ne \emptyset$. 

    \end{remark}

\begin{remark}
    A left infinite string $s$ in $\Str(a,b)$ is a brick if and only if the right infinite string $s^T$ is a brick.
\end{remark}

We end this section by adapting the results of the previous section about strings in $\Str(a,b)$. We have the following theorem.

\begin{theorem}
Let $A = kQ/I$ be a gentle algebra and assume that $a,b$ are finite strings starting and ending at $x$ and satisfying the above conditions. Let $w \in \Str(a,b)$ be infinite. Then
\begin{enumerate}
    \item If $w$ is one-sided infinite, then $M(w)$ is a brick if and only if $w$  or $w^T$ is a characteristic Sturmian word.
    \item If $w$ is double-infinite, then $M(w)$ is a brick if and only if $w$ is a Sturmian word.
\end{enumerate}
\end{theorem}

\textbf{Acknowledgments.}
The authors would like to thank Aaron Chan and Hugh Thomas, for some discussions at the earlier stages of this project, and  Sébastien Labbé, for providing some good references on the combinatorics of words. 
KM was supported by Early-Career Scientist JSPS Kakenhi grant number 24K16908. CP was supported by the Natural Sciences and Engineering Research Council of Canada (RGPIN-2018-04513) and by the Canadian Defence Academy Research Programme.

\end{document}